\documentclass[12pt, reqno]{amsart}
\usepackage{amsmath,amsthm,amscd,amsfonts,amssymb,color}
\usepackage[bookmarksnumbered,colorlinks,plainpages]{hyperref}
\usepackage[all]{xy}
\usepackage{xcolor}
\input{amssym.def}
\input{amssym.tex}
\usepackage{tikz-cd}
\usepackage{graphicx}
\graphicspath{ {./images/} }

\newtheorem{theorem}{Theorem}[section]
\newtheorem{lemma}[theorem]{Lemma}
\newtheorem{proposition}[theorem]{Proposition}
\newtheorem{corollary}[theorem]{Corollary}
\newtheorem{definition}{Definition}
\newtheorem{example}{Example}
\newtheorem{remark}{Remark}
\numberwithin{equation}{section}
\def\DJ{\leavevmode\setbox0=\hbox{D}\kern0pt\rlap
	{\kern.04em\raise.188\ht0\hbox{-}}D}

\usepackage{tikz}
\usetikzlibrary{arrows,shapes,positioning,shadows,trees}
\usepackage[utf8]{inputenc}
\usepackage[T1]{fontenc}
\usepackage{pgf,pgfarrows,pgfnodes,pgfautomata,pgfheaps,pgfshade}
\usepackage{lmodern}
\begin{document}
	\title[Unbounded order demi Dunford-Pettis operators]{Unbounded order demi Dunford-Pettis operators}

	\author{N. Hafidi}
	\author{J. H'michane}
	
	\address{Noufissa Hafidi, Universit\'{e} Moulay Ismail, Facult\'{e} des Sciences, D\'{e}partement de Math\'{e}matiques, B.P. 11201 Zitoune, Mekn\`{e}s, Morocco.}
	\email{hafidinoufissa@gmail.com}

	\address{Jawad H'michane, Engineering Sciences Lab, ENSA, B.P 241, Ibn Tofail University, Kénitra, "Moroccan Association of Ordred Structures, Operators Theory, Applications and Sustainable Development (MAOSOTA)", Morocco.}
	\email{hm1982jad@gmail.com}
	
	\begin{abstract}
		In this paper , we introduce and study a new concept of unbounded order demi Dunford-Pettis operators. Namely, we investigate some properties for this new class of operators and we study its connection with other known operators, we also establish important results about power unbounded order demi Dunford-Pettis operators, and domination property.
		
		%	An operator $T:E\longrightarrow E$ acting on a Banach lattice $E$ is called unbounded order demi Dunford-Pettis if, for every weakly null sequence $(x_n)$ in $E$ such that $(x_n-T(x_n))$ unbounded order converges to $0$ as $n\rightarrow \infty$ we have $(x_n)$ norm converges to $0$  
		
	\end{abstract}
	
	\maketitle
	
	\section{Introduction }
	
	An operator  $T : X\longrightarrow X$ on a Banach space $X$ is  called demicompact, if for every bounded sequence $(x_n)$ in
	the domain $X$ such that $(x_n-T(x_n))$ norm converges to $x$ in  $X$, there is a convergent
	subsequence of $(x_n)$. The concept of demicompactness was first presented and studied by Petryshyn in \cite{PE} and had been greatly expanded by other different studies \cite{JER,JER1}. Demi compact operators considered as generalization of compact operators and defined to solve some fixed point concerns. %Recall from \cite{PE} that an operator $T : D(T)\subseteq X\longrightarrow X$, where $D(T)$ is a
	% subspace of $X$, is said to be demicompact if, for every bounded sequence $(x_n)$ in
	%	 the domain $D(T)$ such that $(x_n-T(x_n))$ norm converges to $x$ in  $X$, there is a convergent
	% subsequence of $(x_n)$.
	
	In their paper  \cite{JER3}, and following the proceeding of generalization, Benkhaled et
	al. introduced the notion of demi Dunford-Pettis operators, regarded as generalization of Dunford-Pettis operators. Recall from \cite{JER3} that an operator $T: X\longrightarrow X$ is said to be demi Dunford–Pettis, if for every sequence $(x_n)$ in $X$ such that  $x_{n} {\overset{w}{\rightarrow}} 0$ and
	$\|x_n-T(x_n)\|\rightarrow 0$, we have $\|x_n\|\rightarrow 0$.
	
	Furthermore, several recent papers focused on the  concept "Demi", and introduced some operators on Banach lattices related to that (\cite{JER2,BEN1,BEN2,GOC,GOC1}).
	
	A net $(x_\alpha)$ of a Banach lattice $E$ is unbounded order convergent (unbounded norm convergent) to $x$, if $|x_{\alpha}-x|\wedge y {\overset{0}{\rightarrow}} 0$ ( if $\||x_{\alpha}-x|\wedge y \|\rightarrow 0$) for each $y\in E^+$. We denote this convergence
	by $x_{\alpha} {\overset{uo}{\rightarrow}} x$ ($x_{\alpha} {\overset{un}{\rightarrow}} x$) and we write  $(x_{\alpha})$ uo-converges to $x$ ($(x_{\alpha})$ un-converges to $x$).   Unbounded order convergence has recently been considerably studied  in many works \cite{UO1,UO2}.
	
	The pupose of the prensent work is to pursue the study of operators on Banach lattices associated to the "Demi" approach. For this aim, we will introduce and study new class of operators that we called unbounded order demi-Dunford-Pettis operators.	An operator $T:E\longrightarrow E$ on a Banach lattice $E$  is said to be unbounded order demi Dunford-Pettis (abb, uo-demi Dunford–Pettis), if for every weakly null sequence $(x_n)$ in $E$ such that $(x_n-T(x_n))$ is uo-convergent to $0$, then $(x_n)$ is norm convergent to $0$.

	The paper is organized as follows. In section \ref{sec1}, we will define the notion of unbounded order demi-Dunford-Pettis operators (see Definition \ref{def}), after we present the main result of this section which is the characterization of unbounded order demi-Dunford-Pettis operators  in terms of Cauchy sequences (see propositions \ref{WS}, \ref{WS1}). In section \ref{sec2}, we discuss the relationships between unbounded order demi-Dunford-Pettis operators and some known classes of operators. Namely, the demi Dunford-Pettis operators (see Propositions \ref{11}, \ref{12}), the un-demi Dunford-Pettis operators (see Proposition \ref{rel1}). We continue in the section \ref{sec5}  by investigating the stability of the above indicated class of  operators under composition  with arbitary bounded operator and  under multiplication by scalar (see Proposition \ref{composition} and Proposition \ref{multi}). We also provide intersting results about the sum of  unbounded order demi-Dunford-Pettis operators (see Proposition \ref{power}). Further, we establish more properties about the power of unbounded order demi-Dunford-Pettis operators. In section \ref{sec7},  We conclude our study  by considering the domination problem for this  class of operators (see Proposition \ref{dom}).

	To state our results, we need to fix some notations and recall some definitions.\	
	Throughout this paper $X$ and $Y$ will denote real Banach spaces, $E$ and $F$ will denote real Banach lattices. A Banach lattice is a Banach space $(E,\Vert \cdot \Vert )$ such that $E$ is a vector lattice and its norm satisfies the following property: for each $x,y\in E$ such that $|x|\leq |y|$, we have $\Vert x\Vert \leq \Vert y\Vert $. If $E$ is a Banach lattice, its topological dual $E^{\prime }$, endowed with the dual norm, is also a Banach lattice. A norm $\Vert \cdot \Vert $ of a Banach lattice $E$ is order continuous if for each net $(x_{\alpha })$ such that $x_{\alpha }\downarrow 0$ in $E$, $(x_{\alpha })$ converges to $0$ for the norm $\Vert \cdot \Vert $, where the notation $x_{\alpha }\downarrow 0$ means that $(x_{\alpha })$ is decreasing, its infimum exists and $\inf
	(x_{\alpha })=0$.  For each $ x, y \in E$  with $x\leq y$, the set $\left[  x; y\right]  = \left\lbrace z\in E: x\leq z \leq y \right\rbrace $ is called
	an order interval. A subset of $E$ is said to be order bounded if it is included in some order
	interval. Also, a vector lattice $E$ is Dedekind $\sigma$-complete if every majorized countable nonempty subset of $E$ has a supremum. A nonzero element $x$ of a vector lattice $E$ is discrete (atom) if the order ideal generated by $x$ equals the vector subspace generated by $x$. The vector lattice $E$ is discrete (atomic), if it admits a complete disjoint system of discrete elements. A Banach space $X$ is said to have the Schur property if every weakly null
	sequence in $X$ is norm null. A Banach lattice $E$ has the positive
	Schur property, if each weakly null sequence with a positive terms in $E$ is norm null. It is pointed out in Theorem 3.7 \cite{GTX}  that $E$ has the positive Schur property 
	if, and only if each weakly null and uo-null sequence in $E$ converges
	to zero in norm. 
	
	We will use the term operator to mean a bounded linear mapping. A linear mapping $T$ from a vector lattice $E$ into another $F$ is order bounded if it carries order bounded set of $E$ into order bounded set of $F$. A linear mapping between two vector lattices $E$ and $F$ is positive if $T(x)\geq 0$ in $F$ whenever $x\geq 0$ in $E$.
	
	Recall from \cite{GOC1} that an operator  $T:E\longrightarrow E$ on a Banach lattice $E$ is called
	unbounded demi Dunford-Pettis (abb, un-demi Dunford–Pettis), if for every weakly null sequence $(x_n)$ in $E$ such that $(x_n-T(x_n))$ is  unbounded norm convergent  to $0$ as $n\rightarrow \infty$, then $(x_n)$ is  unbounded norm converges to $0$ as $n\rightarrow \infty$.

	\section{Unbounded order demi Dunford-Pettis operators.} \label{sec1}
	We start, by giving the definition of  unbounded order demi Dunford-Pettis operators.
	\begin{definition} \label{def}
		An operator $T:E\longrightarrow E$ on a Banach lattice $E$  is called unbounded order demi Dunford-Pettis (abb, uo-demi Dunford–Pettis), if for every weakly null sequence $(x_n)$ in $E$ such that $(x_n-T(x_n))$ is uo-convergent to $0$, then $(x_n)$ is norm convergent to $0$. The collection of unbounded order demi Dunford-Pettis will be denoted by $\mathcal{L}^{uo}_{ddp}(E)$.
	\end{definition}
	
	The following immediate result  characterize the Banach lattices on which all operators are  uo-demi Dunford–Pettis.
	\begin{proposition} \label{SC}
		For a Banach lattice $E$, the following statements are equivalent:
		\begin{enumerate}
			\item	Every operator $T:E\longrightarrow E$  is uo-demi Dunford–Pettis.
			\item The identity operator  $Id_{E}:E\longrightarrow E$ is uo-demi Dunford–Pettis.
			\item  $E$ has the Schur property.
		\end{enumerate}
	\end{proposition}
	\begin{proof}
		Obvious
		
	\end{proof}

	\begin{proposition} \label{22}
		For a Banach lattice $E$, the following statements are equivalent:
		\begin{enumerate}
			\item  For all $\alpha\neq 1$, $\alpha Id_{E}$ is uo-demi Dunford–Pettis.
			\item $-Id_{E}$ is uo-demi Dunford–Pettis.
			\item  $E$ has the positive Schur property.
		\end{enumerate}
	\end{proposition}

	\begin{proof}
		$(1)\Rightarrow (2)$ Obvious. %It suffies to take $\alpha=-1$ 

		$(2)\Rightarrow (3)$ Assume that $x_{n} {\overset{w}{\rightarrow}} 0$ and $x_{n} {\overset{uo}{\rightarrow}} 0$. Now, as $x_n-(-Id_{E}(x_n))=2x_n {\overset{uo}{\longrightarrow}} 0$, and by the fact that $-Id_{E}$ is uo-demi Dunford–Pettis, we have that $\|x_{n}\|\rightarrow 0$. By Theorem 3.7 \cite{GTX} it follows that  $E$ has the positive Schur property.
		
		$(3)\Rightarrow (1)$ We need to show that the operator $\alpha Id_{E}$ is uo-demi Dunford–Pettis for all $\alpha\neq 1$. To this end,  let $x_{n} {\overset{w}{\rightarrow}} 0$ and $x_n-(\alpha Id_{E}(x_n))=(1-\alpha)x_n {\overset{uo}{\longrightarrow}} 0$, obviously $x_n {\overset{uo}{\longrightarrow}} 0$. Now, since $E$ has the positive Schur property, the Theorem 3.7 \cite{GTX} asserts that $\|x_{n}\|\rightarrow 0$, which shows that $\alpha Id_{E}$ is uo-demi Dunford–Pettis.

	\end{proof}	
	
	\begin{remark}
		The condition $\alpha\neq 1$ is essential in the Proposition \ref{22}. Indeed,
		the Banach lattice	$L^{1}[0,1]$ has the positive Schur Property. Howerver,
		the identity operator $Id_{L^{1}[0,1]}:L^{1}[0,1]\rightarrow L^{1}[0,1]$ is not uo-demi Dunford–Pettis. Otherwise,
		let us consider the
		sequence $(r_n)$ of Rademacher functions on $[0, 1]$, we have that $r_{n} {\overset{w}{\rightarrow}} 0$ in $L^{1}[0,1]$ and it is clear that $r_{n}-r_{n} {\overset{uo}{\rightarrow}} 0$. As
		$Id_{L^{1}[0,1]}$ is uo-demi Dunford–Pettis, it follows that $\|r_{n}\|\rightarrow 0$, which is impossible.  
	\end{remark}

	%\section{Characterization in term of weak Cauchy-sequences}
	It is known that unbounded order convergence is not topological. But by Theorem 5.3 \cite{DOT}, we have the following result:
	
	\begin{proposition} \label{uo-un}
		Let $E$ be an atomic order continuous Banach lattice. For any bounded sequence $(x_n)$ of $E$,  the following assertions are equivalent:
		
		\begin{enumerate}
			\item  $(x_n)$ is uo-null. 
			\item  $(x_n)$ is un-null.
			
			\item Every subsequence of  $(x_n)$ is uo-null.
		\end{enumerate}
		
	\end{proposition}	
	In the following, we establish a characterization of uo-demi Dunford–Pettis operators in terms of weak Cauchy sequences.
	\begin{proposition} \label{WS}
		Let $E$ be an atomic  order continuous Banach lattice, and let $T:E\longrightarrow E$ be an operator. Then, the following assertions are equivalent:
		\begin{enumerate}
			\item $T$ is uo-demi Dunford–Pettis.
			\item For every weak Cauchy sequence $(x_n)$ in $E$ such that every subsequence of the  sequence $(x_n-T(x_n))$ is uo-null, we
			have $(x_n)$  is norm convergent.
		\end{enumerate}
	\end{proposition}
	
	\begin{proof} For $(1)\Rightarrow (2)$ Let  $(x_n)$ be a weak Cauchy
		sequence of $E$  such that every subsequence of the sequence $(x_n-T(x_n))$ is uo-null. By way of contradiction, we assume that $(x_n)$  is
		not a norm Cauchy sequence of $E$, then there exist some $\varepsilon>0$ and a subsequence $(x_{n_k} )$ 
		of $(x_n)$  satisfying $ \|x_{n_{k+1}}-x_{n_k} \|>\varepsilon$ for all $k$. So, we have that $ x_{n_{k+1}}-x_{n_k}{\overset{w}{\rightarrow}} 0$, and by hypothesis $ x_{n_{k+1}}-x_{n_k}-T(x_{n_{k+1}}-x_{n_k}) {\overset{uo}{\rightarrow}} 0$. Now, the fact that $T$ is uo-demi Dunford–Pettis yields
		$ \|x_{n_{k+1}}-x_{n_k} \|\rightarrow 0 $, which is impossible. Therefore,  $(x_n)$   is a norm Cauchy sequence, and hence it is
		norm convergent in $E$.
		
		For $(2)\Rightarrow (1)$ Let $(x_n)$ be a sequence of $ E $ such that  $x_n {\overset{w}{\rightarrow}} 0$ and $x_n-T(x_n){\overset{uo}{\rightarrow}} 0$, we have to show that $\|x_{n}\|\rightarrow 0$, it is clear that $(x_n)$ is a weak Cauchy
		sequence of $E$, and Proposition \ref{uo-un} asserts  that every subsequence of the sequence $(x_n-T(x_n))$ is uo-null. In view of (2), we obtain that  $(x_n)$ is
		norm convergent. Thus, $T$ is uo-demi Dunford-Pettis.
		
	\end{proof}

	Using the same argument on  the above proof, we establish the following result.
	
	\begin{proposition} \label{WS1}
		Let $E$  be a  Banach lattice with order continuous norm, and let $T:E\longrightarrow E$ be a operator. Then, the following assertions are equivalents:
		\begin{enumerate}
			\item $T$ is uo-demi Dunford–Pettis.
			\item For every weak Cauchy sequence $(x_n)$ in $E$ such that the sequence $(x_n-T(x_n))$ is un-null, we
			have $(x_n)$  is norm convergent.
		\end{enumerate}
	\end{proposition}
	\iffalse	
	\begin{proof} For $(1)\Rightarrow (2)$ Let  $(x_n)$ be a weak Cauchy
		sequence of $E$  such that every subsequence of the sequence $\{x_n-T(x_n)\}$ is un-null. By way of contradiction, we assume that $\{x_n\}$  is
		not a norm Cauchy sequence of $E$, then there exist some $\varepsilon>0$ and a subsequence $\{x_{n_k} \}$ 
		of $\{x_n\}$  satisfying $ \|x_{n_{k+1}}-x_{n_k} \|>\varepsilon$ for all $k$. Now, since $ x_{n_{k+1}}-x_{n_k}{\overset{w}{\rightarrow}} 0$, and by hypothesis $ x_{n_{k+1}}-x_{n_k}-T(x_{n_{k+1}}-x_{n_k}) {\overset{un}{\rightarrow}} 0$, by passing to further subsequence, we may assume that  $ x_{n_{k+1}}-x_{n_k}-T(x_{n_{k+1}}-x_{n_k}) {\overset{uo}{\rightarrow}} 0$. As $T$ is uo-demi Dunford–Pettis, it follows that
		$ \|x_{n_{k+1}}-x_{n_k} \|\rightarrow 0 $, which is impossible. Therefore,  $\{x_n\}$   is a norm Cauchy sequence, and hence
		norm convergent in $E$.
		
		For $(2)\Rightarrow (1)$ Let  $\{x_n\}$  be a sequence of $ E $ such that  $x_n {\overset{w}{\rightarrow}} 0$ and $x_n-T(x_n){\overset{uo}{\rightarrow}} 0$. Now,  it is clear that $\{x_n\}$ is weak Cauchy
		sequence of $E$, and  since $E$ has order continuous norm, it follows that $x_n-T(x_n){\overset{un}{\rightarrow}} 0$. Hence, $\{x_n\}$ is
		norm convergent to $0$, that proves that $T$ is uo-demi Dunford-Pettis. In the following we investigate relationship between 
		
	\end{proof}
	\fi
	\section{Relationship between uo-demi Dunford-Pettis and other operators} \label{sec2}
	
	Recall from \cite{JER2}, that  an operator $T: X\longrightarrow X$ on a Banach space $X$ is said to be demi Dunford–Pettis, if for every sequence $(x_n)$ in $X$ such that  $x_{n} {\overset{w}{\rightarrow}} 0$ and
	$\|x_n-T(x_n)\|\rightarrow 0$, we have $\|x_n\|\rightarrow 0$.
	
	In the following we investigate the relationship between  demi Dunford–Pettis and uo-demi Dunford-Pettis operators.

	\begin{proposition} \label{11}
		Let $E$ be a Banach lattice. Every  uo-demi Dunford-Pettis  operator $T:E\longrightarrow E$  is demi Dunford-Pettis. 
		
	\end{proposition}
	
	\begin{proof}
		Let $(x_n)$ be a sequence of $E$ such that  $x_n {\overset{w}{\rightarrow}} 0$ and $\|x_n-T(x_n)\|\rightarrow 0$, then there exist $(x_{n_k})$ a subsequence of $(x_n)$  such that  $x_{n_k}-T(x_{n_k}) {\overset{uo}{\rightarrow}} 0$. By hypothesis $T$ is uo-demi Dunford-Pettis, so $\|x_{n_k}\|\rightarrow 0$, and therefore $\|x_{n}\|\rightarrow 0$. %Now, by the same argument, we prove that every  subsequence of $(x_n)$ has a further subsequence which is norm convergent to $0$, and therefore $\|x_{n}\|\rightarrow 0$.
		
	\end{proof}

	\begin{remark}
		The converse of the Proposition \ref{11} is not true in general. Indeed, let us consider the operator $\widetilde{T}:\widetilde{E}\longrightarrow \widetilde{E}$ defined via the matrix:
		$$\begin{pmatrix}
			0 & 0\\ 
			T & Id_{\ell^1}
		\end{pmatrix}$$
		Where, $\widetilde{E}=\ell^{\infty}\times\ell^{1}$ and $T$ an operator from $\ell^{\infty}$ into $\ell^{1}$.
		We will show that  $\widetilde{T}$ is Dunford-Pettis (and hence demi Dunford-Pettis by (Proposition 2.2 \cite{BEN1})). To this end, let $\tilde{x}_n=(x_n,y_n)$ be a sequence of $\widetilde{E}$ such that  $\tilde{x}_n=(x_n,y_n){\overset{w}{\rightarrow}} 0$ in $\widetilde{E}$. By a simple calculation, we obtain: $\|\widetilde{T}(\tilde{x}_n)\|_{\widetilde{E}}=\|T(x_n)+ Id_{\ell^1}(y_n)\|_{\ell^1}\leq \|T(x_n)\|_{\ell^1}+\|y_n\|_{\ell^1}\rightarrow 0$, therefore   $\widetilde{T}$ is Dunford-Pettis. Now, we will establish  that $\widetilde{T}$ is not uo-demi Dunford-Pettis, let  $\tilde{x}_n=(e_n,0)$ where the sequence $(e_n)$ is the unit basis of $\ell^{1}$. Clearly, $\tilde{x}_n{\overset{w}{\rightarrow}} 0$ in  and $\tilde{x}_n-\widetilde{T}(\tilde{x}_n)=(e_n,-T(e_n){\overset{uo}{\rightarrow}} 0$, but  $\|\tilde{x}_n\|_{\widetilde{E}}=\|e_n\|_{\ell^1}\nrightarrow 0$.
		
	\end{remark}
	\iffalse	
	The second example (choose one of them!!!)	
	
	\begin{remark}
		The converse of the Proposition \ref{11} is not true in general.
		There exists an operator  $T:E\rightarrow E$ which is demi Dunford-Pettis, but fails to be uo-demi Dunford-Pettis. Indeed, let $Id_{\ell^\infty} :\ell^\infty\rightarrow \ell^\infty$ be the identity operator of $\ell^\infty$. It is clear that $-Id_{\ell^\infty}$  is demi
		Dunford–Pettis. However by using Proposition \ref{22}, the operator $-Id_{\ell^\infty}$ is not uo-demi Dunford–Pettis.
		
	\end{remark}
	\fi
	Nevertheless, we obtain  this result;
	\begin{proposition} \label{12}
		For a Banach lattice $E$, such that $E$ has the positive Schur property. Every demi Dunford-Pettis (hence Dunford-Pettis) operator $T:E\longrightarrow E$ is uo-demi Dunford-Pettis. 	
	\end{proposition}

	\begin{proof}
		Let $(x_n)$ be a sequence of $E$ such that $x_{n} {\overset{w}{\rightarrow}} 0$ and $x_n-T(x_n){\overset{uo}{\rightarrow}} 0$. We know that $x_n-T(x_n){\overset{w}{\rightarrow}} 0$, and since $E$ is having the positive Schur property, Theorem 3.7 \cite{GTX} asserts that $\|x_n-T(x_n)\|\rightarrow 0$. The fact that $T$ is demi Dunford-Pettis operator yields  $\|x_{n}\|\rightarrow 0$, therefore $T$ is uo-demi Dunford-Pettis.	
		
	\end{proof}
	
	In the following, we give the  necessary and sufficient condition under which each operator is uo-demi Dunford–Pettis.
	
	\begin{proposition}
		For  a Banach lattice $E$, the following are equivalent:
		\begin{enumerate}
			\item Every operator $T:E\longrightarrow E$ is uo-demi Dunford-Pettis
			\item Every operator $T:E\longrightarrow E$  is demi Dunford-Pettis
			\item $E$ has the  Schur property.
			
		\end{enumerate}
	\end{proposition}
	\begin{proof}
		$(1)\Rightarrow (2)$ See Proposition \ref{11}.
		$(2)\Rightarrow (3)$ Follows from Theorem 2.4 \cite{JER2}. For $(3)\Rightarrow (1)$ See Proposition \ref{SC}. 
		
	\end{proof}

	In the next, we investigate the relationship between uo-demi Duford-Pettis operators and unbounded norm demi Dunford-Pettis operators.	
	
	Recall from \cite{NH} that	an operator $T:E\longrightarrow E$ on a Banach lattice $E$ is called unbounded norm demi Dunford-Pettis (abb, un-demi Dunford-Pettis), if for every sequence $(x_n)$ in $E$ such that $(x_n)$ is weakly null and $(x_n-T(x_n))$ un-converges to $0$, then $(x_n)$  un-converges to $0$. 
	
	\begin{lemma} \label{wuo}
		Let $E$ be a Banach lattice with the Positive Schur property.  If  $(x_n)$  is a sequence of $E$ such that $x_{n} {\overset{w}{\rightarrow}} 0$ and $x_n{\overset{un}{\rightarrow}} 0$, then $\|x_n\| \rightarrow 0$.
	\end{lemma}

	\begin{proof}
		Follows from Theorem 3.7 \cite{GTX}.
	\end{proof}

	Note that  un-demi Dunford-Pettis  operator need not to be demi Dunford-Pettis, the converse is also not true in general. Indeed,  the identity operator $Id_{c_0}:c_0\rightarrow c_0$ is clearly un-demi Dunford-Pettis ( proposition 6.2 \cite{GTX}), but fails to be demi Dunford-Pettis. On the other hand, let us consider the operator $\widetilde{T}:\widetilde{E}\longrightarrow \widetilde{E}$ defined via the matrix:
	$$\begin{pmatrix}
		0 & 0\\ 
		T & 2Id_{\ell^{\infty}}
	\end{pmatrix}$$
	Where, $\widetilde{E}=\ell^{1}\times\ell^{\infty}$ and $T$ an operator from $\ell^{1}$ into $\ell^{\infty}$, it follows from Proposition 2.9 \cite{BEN1} that $\widetilde{T}$ is demi Dunford-Pettis operators. But, $\widetilde{T}$  is not un-demi Dunford-Pettis. Indeed, let   $\tilde{x}_n=(0,2e_n)$ where the sequence $(e_n)$ is the unit basis of $\ell^{\infty}$. Clearly $\tilde{x}_n{\overset{w}{\rightarrow}} 0$ in $\widetilde{E}$ and $\|\tilde{x}_n-\widetilde{T}(\tilde{x}_n)\|=0$, therefore $\tilde{x}_n-\widetilde{T}(\tilde{x}_n){\overset{un}{\rightarrow}} 0$ but  $\|\tilde{x}_n\|_{\widetilde{E}}=\|2e_n\|_{\ell^{\infty}}\nrightarrow 0$, hence $\|\tilde{x}_n\|_{\widetilde{E}}{\overset{un}{\nrightarrow}} 0$ . Thus, $\widetilde{T}$  is not un-demi Dunford-Pettis.
	
	We have the following result:

	\begin{proposition} \label{rel1}
		Let $E$ be a Banach lattice such that $E$ has the positive Schur property. Then, the operator $T:E\longrightarrow E$ is uo-demi Dunford-Pettis if, and only if $T:E\longrightarrow E$ is un-demi Dunford-Pettis. 	
	\end{proposition}
	
	\begin{proof}
		For the "only if" part, let $(x_n)$ be a sequence of $E$ such that $x_{n} {\overset{w}{\rightarrow}} 0$ and $x_n-T(x_n){\overset{un}{\rightarrow}} 0$, therefore there exits $(x_{n_k})$ a subsequence of $(x_{n})$ such that $x_{n_k}-T(x_{n_k}){\overset{uo}{\rightarrow}} 0$, by hypothesis $T$ is uo-demi Dunford–Pettis, hence $\|x_{n_k}\|\rightarrow 0$. By the same argument, we prove that every subsequence of $(x_{n})$ has a further subsequence which is norm-null, therefore  $\|x_{n}\|\rightarrow 0$, which proves that $T$ is un-demi Dunford-Pettis.
		
		For the "if" part,  let $(x_n)$ be a sequence of $E$ such that $x_{n} {\overset{w}{\rightarrow}} 0$ and $x_n-T(x_n){\overset{uo}{\rightarrow}} 0$,  since $E$ is order continuous, we have $x_n-T(x_n){\overset{un}{\rightarrow}} 0$. Now, from the fact that $T$ is un-demi Dunford-Pettis, it follows that $x_{n}{\overset{un}{\rightarrow}} 0$. Finally, Lemma \ref{wuo} asserts that $\|x_{n}\|\rightarrow 0$, hence $T$ is uo-demi Dunford–Pettis.

	\end{proof}

	\begin{corollary}
		For a Banach lattice   $E$  with  the positive Schur property. The following assertions are equivalent :
		\begin{enumerate}
			\item  $T:E\longrightarrow E$ is demi Dunford-Pettis.
			\item  $T:E\longrightarrow E$ is uo-demi Dunford-Pettis.
			\item  $T:E\longrightarrow E$ is un-demi Dunford-Pettis.
		\end{enumerate}
		
	\end{corollary}
	\begin{proof}
		Follows from Propositions \ref{12} and  \ref{rel1}. 
		
	\end{proof}
	\iffalse	
	To to summarize the connections
	$$	
	\xymatrix@C=1in@R=0.5in{
		\boxed{ demi \ Dunford-Pettis}  \ar @<4pt> [r]^{PSP} \ar @<-4pt> [d]_{atomic \ PSP}
		& \boxed{uo-demi \ Dunford-Pettis} \ar @<-4pt> [d]_{} \ar @<4pt> [l]^{} 
		& 
		\\
		\boxed{Dunford-Pettis}\ar @<-4pt> [u]_{} \ar @<4pt> [r]^{} 
		&  \boxed{un-demi \ Dunford-Pettis}  \ar @<-4pt> [u]_{PSP}  \ar @<4pt> [l]^{atomic \ PSP} 
		& 
		\\
	}
	$$

	\fi

	\section{Structure} \label{sec5}
	
	Our next concern is to study stability of the  uo-demi Dunford-Pettis operators under the taking of composition with an arbitrary bounded operator. 
	
	We observe through the following example, that the composition is not always stable in the set of uo-demi Dunford-Pettis operators.
	\begin{example} \label{EX1}
		
		We consider the  operators  $T=\frac{1}{2}Id_{L^{1}[0.1]}$ and $S=2Id_{L^{1}[0.1]}$.  Proposition \ref{22} asserts that both $T$ and $S$ are uo-demi Dunford-Pettis. However, using Proposition \ref{SC} the product operator $S\circ T=2Id_{L^{1}[0.1]}\circ\frac{1}{2}Id_{L^{1}[0.1]}=Id_{L^{1}[0.1]}$ is not  uo-demi Dunford-Pettis.
	\end{example}
	
	Thereby, it is worth to describe a subclass of uo-demi Dunford-Pettis operators that is stable under composition. For this aim, we need to define new class of operators, that we call weak-to-unbounded order continuous (abb, weak-uo continuous) operators.
	
	\begin{definition}
		An operator  $T$ between two vector lattices $X$ and $Y$ is said to be weak-to-unbounded order continuous operator (abb, weak-uo continuous), if for any net $(x_{\alpha})$ in $X$ such that $x_{\alpha} {\overset{w}{\rightarrow}} 0$ in $Y$, we have $T(x_{\alpha}){\overset{uo}{\rightarrow}} 0$ in $Y$. We denote by $\mathcal{L}_{w-uo}(X,Y)$ the collection of weak-uo continuous operators. And  $T$ is said to be $\sigma$-weak-to-unbounded order continuous operator (abb, $\sigma$-weak-uo continuous), if for any sequence $(x_n)$ in $X$ such that $x_{n} {\overset{w}{\rightarrow}} 0$ in $X$, we have $T(x_n){\overset{uo}{\rightarrow}} 0$ in $Y$. We denote by $\mathcal{L}_{\sigma-w-uo}(X,Y)$ the collection of $\sigma$-weak-uo continuous operators.
		
	\end{definition}
	The following result gives a necessary and sufficient condition under which each operator is weak-uo continuous.
	
	\begin{proposition} \label{w-uo-c}
		For a Banach lattice $E$, the following statements are equivalent;
		\begin{enumerate}
			\item Every operator $T:X\longrightarrow E$ from  an arbitrary Banach space $X$ into $E$ is weak-uo continuous.
			\item  The identity operator  $Id_{E}:E\longrightarrow E$ is weak-uo continuous.
			\item The linear span of the minimal ideals in $E$ is order dense in $E$.
		\end{enumerate}
	\end{proposition}
	\begin{proof}
		Follows from the definition and from Theorem 1 \cite{wk}.
	\end{proof}
	\begin{proposition} \label{w-uo-c}
		For an order continuous Banach lattice $E$ , the following statements are equivalent;
		\begin{enumerate}
			\item Every operator $T:X\longrightarrow E$ from  an arbitrary Banach space $X$ into $E$ is $\sigma$-weak-uo continuous.
			\item  The identity operator  $Id_{E}:E\longrightarrow E$ is weak-uo continuous.
			\item  $E^+$ is the norm closed convex hull of its extremal rays.
			\item $E$ is atomic.
		\end{enumerate}
	\end{proposition}
	\begin{proof}
		$(1)\Rightarrow (2)$ Obvious.
		
		$(2)\Rightarrow (3)$ Since $E$ is order continuous,	the statement follows  from Theorem 4 \cite{wk}.

		$(3)\Rightarrow (4)$ Assume by way of contradiction that $E$ is not atomic, then by Theorem 2.2 \cite{WIC},  for $x\in E^{+}$, there
		exists a sequence $(x_n)$ such that $x_{n} {\overset{w}{\rightarrow}} 0$ yet $|x_n| = x$ for all $n$. It is clear that $x_{n} {\overset{uo}{\nrightarrow}} 0$, this contradicts the second statement of Theorem 4 \cite{wk}.	
		
		$(3)\Rightarrow (4)$ Let $T:X\longrightarrow E$ be an operator from  a Banach space $X$ into $E$, and let $(x_n)$ be a sequence in $E$ such that $x_{n} {\overset{w}{\rightarrow}} 0$, then  $T(x_{n}) {\overset{w}{\rightarrow}} 0$ by weak continuity of $T$, now as $E$ is atomic and order continuous, Proposition 6.2 and Thoerem 5.3 of \cite{GTX} implie $T(x_{n}) {\overset{uo}{\rightarrow}} 0$, which proves that $T$ is $\sigma$-weak-uo continuous.
	\end{proof}
	
	%	\begin{lemma}
	%	Let $E$ be a Banach lattice such that the linear span of the minimal ideals in $E$ is order dense in $E$. Then, every operator $T:X\longrightarrow E$ from  an arbitrary Banach space $X$ into $E$ is $\sigma$-weak-uo continuous.
	%	\end{lemma}
	
	%	\begin{proof}
	%	Follows from Theorem 1 \cite{wk}
	%	\end{proof}
	
	The following is an important result about  $C(K)$ the
	space of all real-valued continuous functions on $K$ found in \cite{wk}.
	\begin{proposition} [Prop. 2 \cite{wk}] \label{CK}
		Let $K$ be a compact quasi-Stonian space. Every sequence
		in $C(K)$ which converges weakly to $0$ must uo-converge to $0$.
	\end{proposition}
	As consequence of the preceeding proposition, we bring an example of  $\sigma$-weak-uo continuous operators.
	\begin{proposition}
		If $K$ is a compact quasi-Stonian space. Any operator  $T:X\longrightarrow C(K)$ from  an arbitrary Banach space $X$ into $C(K)$ is $\sigma$-weak-uo continuous.
	\end{proposition}
	
	\begin{proof}
		Follows from Proposition \ref{CK}.
	\end{proof}

	%\begin{lemma}
	%	Let $E$ a Banach lattice such that $E$ is atomic and order continuous, Then any operator $T:E\longrightarrow E$ is  weak-uo continuous.
	%\end{lemma}

	%In the next result we show that every un-Dunford-Pettis operator isu $\sigma$o-continuous.
	
	%\begin{proposition}
	%Let $E$ be a Banach lattice. Every un-Dunford-Pettis operator $T:E\longrightarrow E$ is weak-uo continuous. 
	%\end{proposition}
	%\begin{proof}
	%Let $(x_n)$ be a sequence of $E$, such that $x_{n} {\overset{w}{\rightarrow}} 0$, to prove that $T$ is  weak-uo continuous, we show that $T(x_n){\overset{uo}{\rightarrow}} 0$. By, hypothesis $T$ is  un- Dunford-Pettis, therefore  $T(x_{n_k}){\overset{uo}{\rightarrow}} 0$, for some subsequence $(x_{n_k})$ of $(x_{n})$......\textbf{not enough conditions!!!}
	
	%	\end{proof}
	
	Next, we exhibit a result that deals with the  composition in the class of uo-demi Dunford-Pettis  operators.
	\begin{proposition} \label{composition}
		Let $E$ be a Banach lattice, and let  $T, S : E\longrightarrow E$ be two operators. The following statements hold:
		\begin{enumerate}
			\item If $ S \in \mathcal{L}_{\sigma-w-uo}(E)$, then $S$ is uo-demi Dunford-Pettis if and only if   $S\circ T$ is uo-demi Dunford-Pettis.
			\item If both  $T, S \in \mathcal{L}_{\sigma-w-uo}(E)$, then  $T$ is uo-demi Dunford-Pettis if and only if  $S\circ T$ is uo-demi Dunford-Pettis.
			
		\end{enumerate}
		
	\end{proposition} 
	\begin{proof}
		
		For $(1)$ The “only if” part, let $(x_n)$ be a sequence in $E$ such that $x_{n} {\overset{w}{\rightarrow}} 0$ and $x_n-S\circ T(x_n){\overset{uo}{\rightarrow}} 0$. We have that  \begin{equation} \label{equa1}
			x_n- S(x_n)= x_n-S\circ T(x_n)+S[T(x_n)-x_n].
		\end{equation}   Since
		$x_n-T(x_n){\overset{w}{\rightarrow}} 0$ and $S$ is $\sigma$-weak-uo continuous, it follows that $S[T(x_n)-x_n]{\overset{uo}{\rightarrow}} 0$, so from equation \ref{equa1} we have  $x_n-S(x_n){\overset{uo}{\rightarrow}} 0$, the fact that $S$ is  uo-demi Dunford-Pettis yields $\|x_{n}\|\rightarrow 0$, which shows that  $S\circ T$ is uo-demi Dunford-Pettis. For
		the “if” part, we suppose that $x_{n} {\overset{w}{\rightarrow}} 0$ and $x_n-S(x_n){\overset{uo}{\rightarrow}} 0$, using again equation \ref{equa1} we obtain $x_n-S\circ T(x_n){\overset{uo}{\rightarrow}} 0$, and since $S\circ T$ is uo-demi Dunford-Pettis, then $\|x_{n}\|\rightarrow 0$, which proves that $S$ is uo-demi Dunford-Pettis.
		
		For $(2)$  The “only if” part, let $(x_n)$ be a sequence in $E$ such that $x_{n} {\overset{w}{\rightarrow}} 0$ and $x_n-S\circ T(x_n){\overset{uo}{\rightarrow}} 0$. Now, we have that \begin{equation} \label{equa2}
			x_n- T(x_n)= x_n-S\circ T(x_n)+S\circ T(x_n)-T(x_n).
		\end{equation}   Since $x_n{\overset{w}{\rightarrow}} 0$, and
		$T(x_n){\overset{w}{\rightarrow}} 0$ and from the fact that $T, S \in \mathcal{L}_{\sigma-w-uo}(E)$. We deduce from equation \ref{equa2} that  $x_n-T(x_n){\overset{uo}{\rightarrow}} 0$. As $T$ is uo-demi Dunford-Pettis, it follows that $\|x_{n}\|\rightarrow 0$, which shows that  $S\circ T$ is uo-demi Dunford-Pettis. For
		the “if” part, we suppose that $x_{n} {\overset{w}{\rightarrow}} 0$ and $x_n-T(x_n){\overset{uo}{\rightarrow}} 0$, using again equation \ref{equa2} we obtain $x_n-S\circ T(x_n){\overset{uo}{\rightarrow}} 0$, and since $S\circ T$ is uo-demi Dunford-Pettis, then $\|x_{n}\|\rightarrow 0$, which proves that $T$ is uo-demi Dunford-Pettis.
		
	\end{proof}
	
	\begin{remark}
		The weak-uo continuity is essential in Theorem \ref{composition}. Indeed, back to Example \ref{EX1} and considering the operators  $T$ and $S$. We have that $S$ is  uo-demi Dunford-Pettis, but $S$ fails to be $\sigma$-weak-uo continuous, otherwise Proposition \ref{w-uo-c} asserts that $L^{1}[0.1]$ is atomic, which is impossible.  On the other hand,  the  operator $S\circ T$ is not  uo-demi Dunford-Pettis.
	\end{remark}

	The class of uo-demi Dunford-Pettis operators lacks the vector space structure not only
	with the sum but also with the external product. We concretely see that throught the next example:

	\begin{remark} 
		
		\begin{itemize}
			\item[(i)]  Proposition \ref{22} states that $-Id_{L^{1}[0,1]}$  is uo-demi Dunford–Pettis. Howerver,
			the identity operator $Id_{L^{1}[0,1]}$ is not uo-demi Dunford–Pettis.
			
			\item[(ii)] 	Let  consider the following operators $T=-Id_{L^{1}[0,1]}$ and $S=2Id_{L^{1}[0,1]}$. According to Proposition \ref{22}  both $T$ and $S$  are  uo-demi Dunford-Pettis, but the sum $T+S=Id_{L^{1}[0,1]}$ is not  uo-demi Dunford-Pettis. Otherwise, by Proposition  \ref{SC}  $L^{1}[0,1]$ will have the Schur property, which is false.

		\end{itemize}
	\end{remark}
	We establish  the following important result:
	
	\begin{proposition} \label{multi}
		Let $E$ be a Banach lattice and $T:E\longrightarrow E$ an operator. If $T\in \mathcal{L}^{uo}_{ddp}(E)\cap \mathcal{L}_{\sigma-w-uo}(E)$, then  $\alpha T\in \mathcal{L}^{uo}_{ddp}(E)$ for all  $\alpha\in\mathbb{R}$.
	\end{proposition}
	
	\begin{proof}
		Let $(x_n)$ be a sequence of $E$ such that $x_{n} {\overset{w}{\rightarrow}} 0$ and $x_n-\alpha T(x_n){\overset{uo}{\rightarrow}} 0$, for all  $\alpha\in\mathbb{R}$. We have that $x_n-T(x_n)=x_n-\alpha T(x_n)+ \alpha T(x_n)-T(x_n)=x_n-\alpha T(x_n)+T((\alpha-1)x_n)$. $T$ is $\sigma$-weak-uo continuous, hence $T((\alpha-1)x_n){\overset{uo}{\rightarrow}} 0$, and therefore $x_n-T(x_n){\overset{uo}{\rightarrow}} 0$. Now, since $T$ is uo-demi Dunford-Pettis, it follows that  $\|x_{n}\|\rightarrow 0$. Thus, $\alpha T\in \mathcal{L}^{uo}_{ddp}(E)$	
	\end{proof}
	
	%\section{Perturbation}

	The	next result shows  that a $\sigma$-weak-uo continuous operator perturbation of an uo-demi Dunford-Pettis operator is  uo-demi Dunford-Pettis.
	
	%\begin{proposition}
	%	Let $E$ be a Banach lattice and let  $T,S: E\longrightarrow E$ be two operators such that $S$ is weak-uo continuous, then $T+S$ is uo  demi Dunford-Pettis if and only if $T$ is uo  demi Dunford-Pettis.
	%\end{proposition}
	\begin{proposition}
		Let $E$ be a Banach lattice and  $T,S: E\longrightarrow E$ two operators such that $S$ is $\sigma$-weak-uo continuous. Then, the sum $T+S$ is uo-demi Dunford-Pettis if and only if $T$ is uo-demi Dunford-Pettis.
	\end{proposition}
	
	\begin{proof}
		For the "only if" part, assume that $T+S$ is uo  demi Dunford-Pettis, and show that $T$ is uo-demi Dunford-Pettis. To this end, let $(x_n)$ be a sequence of $E$ such that $x_{n} {\overset{w}{\rightarrow}} 0$ and $x_n-T(x_n){\overset{uo}{\rightarrow}} 0$, we have that \begin{equation} \label{eq1}
			x_{n}-[T+S](x_{n})=x_n-T(x_n)-S(x_n)
		\end{equation}
		$S$ is weak-uo continuous, hence $S(x_{n}) {\overset{uo}{\rightarrow}} 0$, and therefore $ x_{n}-[T+S](x_{n}){\overset{uo}{\rightarrow}} 0$. Now, by hypothesis $T+S$ is uo-demi Dunford-Pettis, hence $\|x_{n}\|\rightarrow 0$, this shows that $T$ is  uo-demi Dunford–Pettis.
		
		For the "if" part, assume that  $T$ is uo-demi Dunford–Pettis, and let $(x_n)$ be a sequence of $E$ such that $(x_n){\overset{w}{\rightarrow}} 0$ and $x_{n}-[T+S](x_{n}) {\overset{uo}{\rightarrow}} 0$, we have from (\ref{eq1}) that $x_n-T(x_n)=x_{n}-[T+S](x_{n})+ S(x_{n})$. Since $S(x_{n}) {\overset{uo}{\rightarrow}} 0$ it follows that $x_n-T(x_n){\overset{uo}{\rightarrow}} 0$, and as $T$ is  uo-demi Dunford–Pettis we have $\|x_{n}\|\rightarrow 0$, therefore $T+S$ is  uo-demi Dunford–Pettis.
		
	\end{proof}
	
	As consequence, we obtain a sufficient condition for which sum of two uo-demi Dunford-Pettis operators is uo-demi Dunford-Pettis.
	\begin{corollary}
		Let $E$ be a Banach lattice and let $T,S\in \mathcal{L}^{uo}_{ddp}(E)$. If $S\in\mathcal{L}_{\sigma-w-uo}(E)$ or $T\in\mathcal{L}_{\sigma-w-uo}(E)$, then  the sum $T+S$ is  uo-demi Dunford–Pettis.	
	\end{corollary}

	\iffalse	
	
	Note that uo-demi Dunford-Pettis operators need not to be  un demi Dunford-Pettis, and converse is not either true.
	
	The following result shows that an uo-demi Dunford-Pettis operator perturbed with un-Dunford-Pettis is  uo-demi Dunford-Pettis.
	
	\begin{proposition} \label{P2}
		Let $T: E\longrightarrow E$  be an uo-demi Dunford-Pettis operator. If $S: E\longrightarrow E$ is un-Dunford-Pettis, then the operator $ S+T$ is un demi Dunford-Pettis.
	\end{proposition}
	\begin{proof} 
		Let $(x_n){\overset{w}{\rightarrow}} 0$ and $x_{n}-[T+S](x_{n}) {\overset{un}{\rightarrow}} 0$. Since, $S$ is un- Dunford-Pettis, $S(x_{n}) {\overset{un}{\rightarrow}} 0$. Now, since  $x_{n}-T(x_{n}) =x_{n}-[T+S](x_{n})+ S(x_{n})$, it follows   $x_{n}-T(x_{n}) {\overset{un}{\rightarrow}} 0$, so there exists a subsequence $(x_{n_k})$ such that $x_{n_k}-T(x_{n_k}) {\overset{uo}{\rightarrow}} 0$. By hypothesis $T$ is  uo-demi Dunford-Pettis, hence $\|x_{n_k}\|\rightarrow 0$ and therefore  $\|x_{n}\|\rightarrow 0$, that proves that $ S+T$ is un demi Dunford-Pettis.
	\end{proof}
	
	\begin{corollary}
		Let $E$ be an order continuous atomic Banach lattice. Let  $T, S: E\longrightarrow E$ be two operators. If $T: E\longrightarrow E$ is uo-Dunford-Pettis, then the operator $ S+T$ is un demi Dunford-Pettis.
	\end{corollary}
	
	\begin{proof}
		Follows from  Proposition 3.1 (un-dp) and Proposition \ref{P2}.
		
	\end{proof}
	\fi	
	\section{Power uo-demi Dunford–Pettis operators}
	In the following section, we will be focusing on power uo-demi Dunford–Pettis operators. But before, we  
	recall from \cite{OUC} that an operator $T: E\longrightarrow F$ between two Riesz spaces is said to be
	unbounded $\sigma$-order continuous (abb, u $\sigma$ o-continuous), if for sequence $(x_{n})$ of $E$ such that  $x_{n} {\overset{uo}{\rightarrow}} 0$ in $E$ implies $T(x_{n}) {\overset{uo}{\rightarrow}} 0$ in $F$.
	\begin{proposition}
		Let $E$ be a Banach lattice, and $T: E\longrightarrow E$ an u $\sigma$o-continuous operator. If $T^{2}$ is uo-demi Dunford–Pettis, then $T$ is uo-demi Dunford–Pettis.	
	\end{proposition}
	
	\begin{proof}
		Let $(x_n)$ be a weakly null sequence in $E$ and $x_{n}-T(x_{n}) {\overset{uo}{\rightarrow}} 0$, since $T$ is u $\sigma$o-continuous,  we have $T(x_{n})-T^{2}(x_{n}) {\overset{uo}{\rightarrow}} 0$. Now, by the equality $ x_n-T^{2}(x_{n})= x_{n}-T(x_{n})+ T(x_{n})-T^{2}(x_{n}) $, we infer that $x_{n}-T^{2}(x_{n}) {\overset{uo}{\rightarrow}} 0$. Now, as $T^{2}$ is uo-demi Dunford–Pettis, it follows that  $\|x_{n}\|\rightarrow 0$, which proves that $T$ is  uo-demi Dunford–Pettis.
		
	\end{proof}

	Generalization of the previous result.
	
	\begin{proposition} \label{power}
		Let $E$ be a Banach lattice. If $T: E\longrightarrow E$ is an u $\sigma$o-continuous power uo-demi Dunford-Pettis operator, then $T$ is uo-demi Dunford–Pettis.	
	\end{proposition}
	
	\begin{proof}
		Assume that $T$ is power uo-demi Dunford-Pettis, that is, if $T^m$ is uo-demi Dunford-Pettis for some $m\in\mathbb{N}$ and show that $T$ is uo-demi Dunford-Pettis. To this end, let $(x_n)$ be a  sequence of $E$ such that $x_n{\overset{w}{\rightarrow}} 0$  and $x_{n}-T(x_{n}) {\overset{uo}{\rightarrow}} 0$. Since, $T$ is u $\sigma$o-continuous, we can easily check  that  $T^{i}(x_{n})-T^{i+1}(x_{n}) {\overset{uo}{\rightarrow}} 0$ for all $i\in\{1,..,m\}$. On the other hand, from the equality 
		
		$$x_{n}-T^{m}(x_{n})=x_{n}-T(x_{n})+\sum_{i=1}^{m-1}T^{i}(x_{n})-T^{i+1}(x_{n})$$
		it follows that $x_{n}-T^{m}(x_{n}) {\overset{uo}{\rightarrow}} 0$. Now, as $T^m$ is uo-demi Dunford-Pettis, we deduce that $\|x_{n}\|\rightarrow 0$, and thus $T$ is  uo-demi Dunford–Pettis.
		
	\end{proof}
	
	\begin{remark}
		
		It is clear that every identity operator $Id_{E}: E\rightarrow E$ is u $\sigma$ o-continuous. we should note that, there exist u $\sigma$ o-continuous operators that are not $\sigma$-weak uo-continuous. Indeed, the identity operator of $\ell^\infty$  $Id_{\ell^\infty} :\ell^\infty\rightarrow \ell^\infty$ is u $\sigma$ o-continuous, but using Proposition \ref{w-uo-c} it fails to be $\sigma$-weak uo-continuous. This fact lead us to the next result important result which is a consequence of Proposition \ref{composition}.

		\begin{proposition}
			Let $E$ be a Banach lattice. If $T: E\longrightarrow E$ is $\sigma$-weak uo-continuous power uo-demi Dunford-Pettis operator, then $T$ is uo-demi Dunford–Pettis.	
		\end{proposition}
		\begin{proof}
			Follows from Proposition \ref{composition}.
			
		\end{proof}
		
	\end{remark}
	
	We obtain also the following result:
	
	\begin{proposition}
		Let $E$ a Banach lattice and  $T: E\longrightarrow E$ an operator.
		If $T$ and $-T$ are both uo-demi Dunford-Pettis, then $T^2$ is uo-demi Dunford-Pettis operator.
	\end{proposition}
	
	\begin{proof}
		To show that  $T^2$ is uo-demi Dunford-Pettis operator, let $x_n{\overset{w}{\rightarrow}} 0$  and $x_{n}-T^2(x_{n}) {\overset{uo}{\rightarrow}} 0$. we have the following
		\begin{align*}
			x_{n}-T(x_{n}) &= x_{n}-T^2(x_{n})+T^2(x_{n})-T(x_{n}) \\
			&= x_{n}-T^2(x_{n})-T[x_{n}-T(x_{n})] \\
			\intertext{we set $z_n=	x_{n}-T(x_{n})$, clearly $z_n{\overset{w}{\rightarrow}} 0$. On the other hand we have that}
			z_{n}+T(z_{n})=z_{n}-[-T(z_{n})]	&= x_{n}-T^2(x_{n}){\overset{uo}{\rightarrow}} 0
		\end{align*}
		By hypothesis $-T$ is  uo-demi Dunford-Pettis, therefore 	 $\|z_n\|=\|x_{n}-T(x_{n})\|\rightarrow 0$. Now, as $T$ is  uo-demi Dunford-Pettis, Proposition \ref{11} shows that $T$ is demi Dunford-Pettis, hence $\|x_{n}\|\rightarrow 0$, which proves that $T^2$ is uo-demi Dunford-Pettis.
		
	\end{proof}
	
	\begin{proposition}
		Let $E$ be a Banach lattice, and $T:E\longrightarrow E$ a  $\sigma$-weak-uo continuous operator. If $T$ is uo-demi Dunford-Pettis operator, then its power $T^m$ for all $m\in\mathbb{N}^*$ is uo-demi Dunford-Pettis.
	\end{proposition}
	
	\begin{proof}
		Follows from statement $(1)$ of Proposition \ref{composition}.
		
	\end{proof}

	%\begin{remark}
	%identity operator of $c_0$ is not uo-demi Dunford–Pettis operator.
	%\end{remark}
	
	\section{Dominitaion problem} \label{sec7}

	We should note that the domination problem is not valid in the class of uo-demi Dunford-Pettis operators, as the following example shows.
	\begin{example}
		Let $ Id_{L^{1}[0.1]}:L^{1}[0.1]\longrightarrow L^{1}[0.1] $ be the identity operator of $L^{1}[0.1]$. We can see that  $0\leq Id_{L^{1}[0.1]}\leq 2 Id_{L^{1}[0.1]}$. Now, by using Proposition \ref{22}, we have that $2 Id_{L^{1}[0.1]}$ is uo-demi Dunford-Pettis. But, Proposition \ref{SC} asserts that $Id_{L^{1}[0.1]}$ is not uo-demi Dunford-Pettis.
	\end{example}
	
	By the following result, we show that the domination problem  is satisfied in the setting of central uo-demi Dunford-Pettis operators. For this aim, we need to recall from \cite{ABRA} that an operator  $T:E\longrightarrow E$ on a Banach lattice $E$ is called  central if it is dominated by a multiple of the identity. That is, $T$ is central operator if and only if there exists some scalar, $\lambda>0$ such that $ |T(x)| \leq \lambda |x|$ holds for all $x\in E$.

	\begin{proposition} \label{dom}
		Let E be a Banach lattice and $S$, $T:E\longrightarrow E$  be a positive
		operators such that, $0 \leq S \leq T \leq Id_{E}$. If  $T$ is uo-demi Dunford-Pettis
		operator, then $S$ is also uo-demi Dunford-Pettis operator, where $Id_{E}$
		denotes the identity operator.
		
	\end{proposition}
	
	\begin{proof}
		Let $x_{n} {\overset{w}{\rightarrow}} 0$ and $x_n-S(x_n){\overset{uo}{\rightarrow}} 0$, we need to show that $\|x_{n}\|\rightarrow 0$. It easily to see that  $|I-S(x)| \leq|x|$  all $x\in E$, hence the operator $I-S$ is central. Using Theorem 3.30 \cite{ABRA}, we obtain 	$|I-S(x_n)|=I-S(|x_n|)$, on the other hand we have; 
		\begin{align*}
			|I-T(x_{n})|\wedge u &\leq I-T(|x_{n}|)\wedge u \\
			& \leq I-S(|x_{n}|)\wedge u \\
			& = |I-S(x_{n})|\wedge u {\overset{o}{\rightarrow}} 0 ,   \text{for all  } u\in E^+
		\end{align*}
		
		Thus, $x_n-T(x_n){\overset{uo}{\rightarrow}} 0$, as $T$ is uo Dunford-Pettis, we infer that $\|x_{n}\|\rightarrow 0$, and so $S$ is uo-demi Dunford-Pettis operator.	
	\end{proof}

	We end by the following important result:
	\begin{proposition}
		Let $E$ be a Banach lattice and $T:E\longrightarrow E$ an uo-demi Dunford-Pettis operator such that $Ker(Id_{E}-T)$ is a Riesz space. Then, $Ker(Id_{E}-T)$ has the positive Schur property when it is considered as a Banach lattice in its own right.
	\end{proposition}
	\begin{proof}
		In view of Corollary \cite{Mey}, it suffices to show that every weakly subset $A$ of $Ker(Id_{E}-T)$ is L-weakly compact. To this, let $A$ be a weakly subset of $Ker(Id_{E}-T)$, and let $(x_n)$ be a disjoint sequence in $Sol(A)$. Then Theorem 4.64 implies $x_{n} {\overset{w}{\rightarrow}} 0$. As $(x_n)\in Sol(A)\subset Ker(Id_{E}-T)$ it follows that $x_n-T(x_n)=0$ for all $n$. Thus, $x_n-T(x_n){\overset{uo}{\longrightarrow}} 0$  all $n$. Now, the fact that $T$ is uo-demi Dunford-Pettis implies that $\|x_{n}\|\rightarrow 0$. This proves that  $A$ is L-weakly compact subset of $Ker(Id_{E}-T)$.
	\end{proof}
	
	%\begin{proposition}
	%		Let $E$ be a Banach lattice and $T:E\longrightarrow E$ a uo-demi Dunford-Pettis operator such that $T$ has distinct eigen values $ \{\lambda_{1},\lambda_{2},...,\lambda_{p} \}$ 
	%.\end{proposition}

\end{document}